\documentclass[11pt]{amsart}

\usepackage{float}
\usepackage{soul}
\usepackage{wrapfig}
\usepackage{tikz-cd}
\usepackage{tikz}
\usepackage{amsthm}
\usepackage{graphicx} 
\usepackage{caption}
\usepackage{subcaption}
\usepackage[backref]{hyperref}
\usepackage{cleveref} 
\usepackage{txfonts}
\usepackage{booktabs}
\usepackage{comment}
\usepackage{enumitem}   

\DeclareMathOperator{\rk}{rk}
\DeclareMathOperator{\Gal}{Gal}
\DeclareMathOperator{\red}{red}

\newcommand{\Q}{\mathbb{Q}}
\newcommand{\Z}{\mathbb{Z}}

\newcommand{\F}{\mathbb{F}}
\newcommand{\OO}{\mathcal{O}}

\newtheorem{theorem}{Theorem}
\newtheorem*{theorem*}{Theorem}
\newtheorem{lemma}[theorem]{Lemma}

\newtheorem*{proposition*}{Proposition}

\newtheorem*{corollary*}{Corollary}

\theoremstyle{definition}
\newtheorem{remark}[theorem]{Remark}

\title{Quadratic points on $X_0(163)$}

\date{\today}

\begin{document}

\begin{abstract}
    We determine all the quadratic points on the genus $13$ modular curve $X_0(163)$, thus completing the answer to a recent question of Banwait, the second-named author, and Padurariu. In doing so, we investigate a curious phenomenon involving a cubic point with complex multiplication on the curve $X_0(163)$. This cubic point prevents us, due to computational restraints, from directly applying the state-of-the-art Atkin--Lehner sieve for computing quadratic points on modular curves $X_0(N)$. To overcome this issue, we introduce a technique which allows us to work with the Jacobian of curves modulo primes by directly computing linear equivalence relations between divisors.
\end{abstract}


\author[Michaud-Jacobs]{Philippe Michaud-Jacobs}
\address{Mathematics Institute, University of Warwick, CV4 7AL, United Kingdom}
\email{p.rodgers@warwick.ac.uk}
\urladdr{\url{https://warwick.ac.uk/fac/sci/maths/people/staff/michaud/}}

\author[Najman]{Filip Najman}
\address{University of Zagreb, Bijeni\v{c}ka Cesta 30, 10000 Zagreb, Croatia}
\email{\url{fnajman@math.hr}}
\urladdr{\url{https://web.math.pmf.unizg.hr/~fnajman/}}


\thanks{P.M. was supported by an EPSRC studentship EP/R513374/1 and has previously used the surname Michaud-Rodgers.
F.N. was supported by the QuantiXLie Centre of Excellence, a project co-financed by the Croatian Government and European Union through the European Regional Development Fund - the Competitiveness and Cohesion Operational Programme (Grant KK.01.1.1.01.0004).}





\maketitle

\section{Introduction}

The problem of determining quadratic points on modular curves $X_0(N)$ has recently received  considerable attention by many authors, with an aim towards classifying the possible isogeny degrees of elliptic curves over quadratic fields, see \cite{akmnov, box21, bruin-najman,najman-vukorepa,ozman-siksek}.

Question~1.3 of \cite{bnp} asks whether one can determine the (finitely many) quadratic points on each of the curves $X_0(91)$, $X_0(125)$, $X_0(163)$, and $X_0(169)$. Answering this question is a key step in computing the possible isogeny degrees of elliptic curves over fixed quadratic fields, which in turn has direct applications to the resolution of Diophantine equations using the modular method. The quadratic points on $X_0(91)$
are determined in \cite[Section~13.6]{vukorepa2022} and the quadratic points on $X_0(125)$ and $X_0(169)$ are determined in \cite[Section~6]{bnp}. The case of $X_0(163)$, a curve of genus $13$, is thus the last remaining case and serves as the chief motivation for writing this note. 

At the time of writing, the state-of-the-art technique for determining quadratic points on modular curves $X_0(N)$ is the Atkin--Lehner sieve (a variant on the classic Mordell--Weil sieve), presented in detail in \cite[Section~3.4]{akmnov}. Due to the existence of a certain complex multiplication (CM) cubic point on the curve $X_0(163)$, it turns out that the Atkin--Lehner sieve, when applied by combining information modulo odd primes $< 41$, is in fact guaranteed to fail for this curve. We provide a complete explanation for this interesting phenomenon in Section~\ref{sec:failure} of this note. The large genus of the curve $X_0(163)$  creates computational issues and prevents us from applying the sieve using primes $\geq 17$ (and certainly from using primes $\geq 41$). This makes this curve particularly difficult to study. To overcome this problem, we work directly with linear equivalence relations between degree $2$ divisors, and thus avoid working with the Jacobian of the curve $X_0(163)$ modulo larger primes.  This allows us to work modulo $41$ and ultimately prove the following result.

\begin{theorem} \label{mainthm}
    The modular curve $X_0(163)$ has nine pairs of quadratic points, all of which are CM points, and each pair of quadratic points is interchanged by the Atkin--Lehner involution. 
    
    In \Cref{table163} below, we have displayed one representative of each pair of quadratic points, its field of definition, its $j$-invariant, and the discriminant of the endomorphism ring of an elliptic curve represented by the point.

    \begin{table}[!ht]
                {\small
        \begin{tabular}{ccccc}
	    Point & Field  & $j$-invariant & CM  \\ [0.3ex]
	    \midrule 
            $P_1$ & $\Q(\sqrt{-2})$ & $8000$ & $-8$  \\ [0.5ex]
            $P_2$ & $\Q(\sqrt{-3})$  & $0$ & $-3$ \\ [0.5ex]
            $P_3$ & $\Q(\sqrt{-3})$  &  $54000$  & $-12$\\  [0.5ex]
            $P_4$ & $\Q(\sqrt{-3})$  &  $-12288000$  & $-27$\\  [0.5ex]
            $P_5$ & $\Q(\sqrt{-7})$ & $-3375$ & $-7$ \\ [0.5ex] 
            $P_6$ & $\Q(\sqrt{-7})$ & $16581375$ & $-28$ \\ [0.5ex] 
            $P_7$ & $\Q(\sqrt{-11})$ & $-32768$ & $-11$ \\ [0.5ex]
            $P_8$ & $\Q(\sqrt{-19})$  & $-884736$ & $-19$ \\ [0.5ex]
            $P_9$ & $\Q(\sqrt{-67})$  & $-147197952000$ & $-67$ \\ [0.5ex]
            \bottomrule
        \end{tabular}
        
    } 
    \caption{All quadratic points on $X_0(163)$.}
    \label{table163}
    \end{table}   
    
\end{theorem}

In the accompanying code repository (linked below), we have displayed the coordinates of each of these quadratic points on a model for the curve $X_0(163)$ in $\mathbb{P}^{12}$.

We expect that the technique we employ to avoid working with the Jacobian of $X_0(163)$ modulo larger primes is applicable more generally for computing low-degree points on curves using any type of Mordell--Weil sieve. We discuss this further in Remark~\ref{end_rem}.

We now briefly outline the remainder of this note. In Section~\ref{sec:difficulties}, we apply the Atkin--Lehner sieve to the curve $X_0(163)$, making some progress towards proving Theorem~\ref{mainthm}. In Section~\ref{sec:failure}, we provide a complete explanation for the failure of the Atkin--Lehner sieve when applied using primes $<41$. Finally, in Section~\ref{sec:41}, we complete the proof of Theorem~\ref{mainthm}.

The \texttt{Magma} \cite{magma} code used to verify the computations in this paper is available at
\begin{center}
    \url{https://github.com/michaud-jacobs/quadratic-points-X0163}
\end{center}


\subsection*{Acknowledgements} We thank Nikola Ad\v zaga and Samir Siksek for helpful conversations. 


\section{Applying the Atkin--Lehner sieve} \label{sec:difficulties}

In this section we introduce notation and apply the Atkin--Lehner sieve to the curve $X_0(163)$ using small primes. We will briefly recall some key details of the Atkin--Lehner sieve, but for a more thorough treatment we refer the reader to \cite[Section~3.4]{akmnov} and \cite[p.~1806]{najman-vukorepa}. 

Let $X \coloneqq X_0(163)$ and let $J \coloneqq J(X)$ be the Jacobian of $X_0(163)$. Let $w \coloneqq w_{163}$ be the Atkin--Lehner involution on $X_0(163)$, and by abuse of notation, we will also denote by $w$ the induced involutions on the symmetric square $X^{(2)}$ of $X$ and on $J$. We view elements of $X^{(2)}$ as effective divisors of degree $2$. Let $c_0$ and $c_\infty$ denote the cusps of $X$. The curve $X_0(163)$ also has a rational CM point (arising from a rational elliptic curve with a rational $163$-isogeny) and we denote this point by $P_{CM}$. By \cite[Theorem~(I)]{eisenstein} the group $J(\Q)_\mathrm{tors}$ is a cyclic group of order 27, generated by $[c_0-c_\infty]$. We define 
\[ 
    D_\infty \coloneqq c_0+c_\infty \quad \text{ and } \quad D_t \coloneqq c_0-c_\infty.
\] 
Note that $w(D_\infty)=D_\infty$ and $w(D_t)=-D_t$, as $w$ permutes the cusps $c_0$ and $c_\infty$. Let $\iota\colon X^{(2)} \hookrightarrow J$ be the Abel--Jacobi map sending a divisor $D$ to $[D-D_\infty]$. As $X$ is not hyperelliptic, $\iota$ is injective. 

In order to apply the sieve, using the code and terminology of \cite{akmnov}, we start by computing a diagonalised model for $X$ together with a map to the quotient curve $X/w = X_0^+(163)$. This quotient curve has genus $6$ and we will write $\pi: X \rightarrow X/w$ for the natural quotient map. We compute that $\rk (1-w)J(\Q)=0$, a necessary condition to apply the sieve. As explained in \cite[Section~3.4]{akmnov} a successful application of the sieve will determine all quadratic points on $X$ that do not arise as pullbacks of rational points under the quotient map $\pi$ and that are not fixed points of $w$. The following two lemmas describe these two sets of points.

\begin{lemma}\label{quotient_points_lemma}
    The curve $X/w$ has precisely $11$ rational points, $9$ of which give rise to pairs of quadratic points on $X_0(163)$ which are displayed in \Cref{table163} of the introduction.
\end{lemma}

\begin{proof}
    The curve $X/w$ has precisely $11$ rational points by \cite[Section~5.3.1]{AABCCKW}, which also lists each non-cuspidal point's corresponding CM order, and this in turn allows us to obtain the $j$-invariant of the point (we may alternatively directly compute the $j$-invariant of each point using our model for the curve). We pull back each of these points via $\pi$. The rational cusp on $X/w$ gives rise to the cusps $c_0$ and $c_\infty$ on $X$ and the rational point with CM by $\Z\left[\frac{1+\sqrt{-163}}{2}\right]$ gives rise to the rational CM point $P_{CM} \in X(\Q)$. The remaining $9$ points give rise to the pairs of quadratic points displayed in \Cref{table163}.
\end{proof}

Before stating the next lemma, which describes the fixed points of $w$ on $X$, we introduce some further notation. We write $K = \Q(\sqrt{-163})$, $\OO_K=\Z\left[\frac{1+\sqrt{-163}}{2}\right]$ for the ring of integers of $K$,  and $\OO \coloneqq \Z \left[\sqrt{-163}\right]$. We then write $B = \Q(j(\OO))$, which is a non-Galois number field of degree $3$. The algebraic integer $j(\OO)$ may be obtained as a root of the Hilbert class polynomial of $\OO$. Alternatively, $B$ may be constructed as $B = \Q(\alpha)$, for $\alpha$ a root of the simpler polynomial $x^3 - 8x + 10$. The field $H \coloneqq K(\alpha)$ is the ring class field of $\OO$ (see \cite[Theorem~11.1]{cox} for example). 

\begin{lemma}\label{fixed_points_lemma}
    The fixed points of $w$ on $X$ defined over number fields are given by:
    \begin{enumerate}[label = (\roman*)]
        \item the rational fixed point, $P_{CM}$, corresponding to an elliptic curve with CM by $\OO_K$;
        \item each member of the Galois orbit of a cubic fixed point defined over $B$, which we denote $R_{CM}$, corresponding to an elliptic curve with CM by $\OO$. 
    \end{enumerate}
\end{lemma}

\begin{proof}
    This follows, for example, from \cite[p.~454]{Ogg}. Alternatively, we may also compute the fixed points of $w$ on $X$ using the explicit model of our curve.
\end{proof}

We have displayed the coordinates of the fixed points of $w$ on our model for $X$ in the accompanying code repository. The $j$-invariant of the point $P_{CM}$ is given by $-262537412640768000$ and the  $j$-invariant of the point $R_{CM}$ is in $B \backslash \Q$ and is displayed in the accompanying code repository.

We are now in a position to apply the sieve. We first briefly recall the setup. For any object $O$ (point, divisor, function, etc.), we will denote by $\red_p(O)$ or $\widetilde O$ its reduction modulo $p$, where $p$ will always be clear from the context. For a given prime $p$ of good reduction for the curve $X$ we have the following commutative diagram:
\begin{equation*} 
    \begin{tikzcd}[sep = large]
    X^{(2)}(\Q) \arrow[hook]{r}{\iota}  \arrow{d}{\red_p} &  J(\Q)  \arrow{d}{\red_p} \arrow{r}{1-w} & J(\Q)_\mathrm{tors} \arrow[hook]{d}{\red_p}
    \\ 
    X^{(2)}(\mathbb{F}_p) \arrow{r}{\tilde{\iota}} & J(\mathbb{F}_p) \arrow{r}{1-\tilde{w}} &  J(\F_p)
\end{tikzcd}\end{equation*}

We now introduce an important definition. We say that a point $Q \in X^{(2)}(\Q_p)$ is \textbf{$p$-adically lonely} if for any point $R \in X^{(2)}(\Q_p)$ satisfying $\mathrm{red}_p(R) = \mathrm{red}_p(Q)$ we have $R = Q$. We may attempt to prove that a point is $p$-adically lonely by applying a symmetric Chabauty criterion, as in \cite[Section~3.4]{akmnov}, although this is not guaranteed to prove that a $p$-adically lonely point is in fact $p$-adically lonely.

We define $H_p \subseteq X^{(2)}(\F_p)$ to be the set of mod $p$ reductions of the known points in $X^{(2)}(\Q)$ that we can prove are $p$-adically lonely. The known points (divisors) in our case are those supported on $\{c_\infty, c_0, P_{CM} \}$ and those appearing in \Cref{table163}. We define the set $S_p$ to be $X^{(2)}(\F_p) \backslash H_p$.
We then define (as in \cite[Section~3.4]{akmnov})
\begin{equation}\label{W_p}
    W_p \coloneqq \red_p^{-1} \big( (1-\tilde{w}) \circ \tilde{\iota} )(S_p) \big) \subseteq J(\Q)_\mathrm{tors}. 
\end{equation}

Thanks to Lemmas~\ref{quotient_points_lemma} and \ref{fixed_points_lemma}, the contents of \cite[Section~3.4]{akmnov} tells us that in order to prove Theorem~\ref{mainthm}, it would suffice to show that 
\[
    \bigcap_{p \in \mathcal{P}} W_p=\{0\} \subseteq J(\Q)_\mathrm{tors}, 
\] 
for a set of primes $\mathcal{P}$ of good reduction for $X$. We say that the sieve \textit{fails} if this condition is not satisfied. As a first step towards proving Theorem~\ref{mainthm} we applied the sieve with the primes $3$ and $5$ to obtain the following result.

\begin{lemma}\label{lem:35}
    Let  $\mathcal{P} = \{ 3,5 \}$. Then  
    \[ 
        \bigcap_{p \in \mathcal{P}}W_p=\{0, [D_t], -[D_t] \} \subset J(\Q)_\mathrm{tors}. 
    \]
\end{lemma}

We note that $\#W_3 = 5$. We also tested the sieve using $\mathcal{P} = \{3,5,7,11,13\}$, but we still found that $[D_t], [-D_t] \in \bigcap_{p \in \mathcal{P}}W_p$. We were not able to use primes larger than $13$ due to both computational time and memory restraints, although we will shortly see that even using primes up to $37$ would not be of any help.


\section{The failure of the Atkin--Lehner sieve for small primes} \label{sec:failure}

In this section we see exactly why the Atkin--Lehner sieve fails to prove Theorem~\ref{mainthm} when applied using primes $<41$. Although the contents of this section does not form part of the proof of Theorem~\ref{mainthm}, it illustrates many subtleties of the Atkin--Lehner sieve, and obstructions to successful applications of Mordell--Weil sieves more generally. Perhaps more importantly for our purposes, the contents of this section will guide us towards the need for incorporating the prime $p = 41$ into our computations, ultimately allowing us to prove Theorem~\ref{mainthm} in Section~\ref{sec:41}.

For the moment, we fix $p \neq 163$ to be an odd prime, and we make two important observations. We first note that 
\begin{equation}\label{D_t_eq}
    ((1-w) \circ \iota)(c_0+P_{CM})=[D_t] \quad \text{and} \quad 
    ((1-w) \circ \iota)(c_\infty+P_{CM})=-[D_t]. 
\end{equation}
Secondly, we see that 
\begin{equation}\label{equally_lonely}
    c_0+P_{CM} \text{ is $p$-adically lonely } \Longleftrightarrow \; c_\infty+P_{CM} \text{ is $p$-adically lonely.}
\end{equation} 
Indeed, if $Q \in X^{(2)}(\Q_p)$ satisfies $\tilde{Q} = \tilde{c}_0 + \tilde{P}_{CM}$ with $Q \neq c_0 + P_{CM}$ (so that $c_0+P_{CM}$ is not $p$-adically lonely) then $w(Q) \neq c_\infty + P_{CM}$ and $\tilde{w}(\tilde{Q}) = \tilde{c}_\infty +\tilde{P}_{CM}$ (so $c_\infty +P_{CM}$ is also not $p$-adically lonely). These two observations will help us prove the following lemma.

\begin{lemma}\label{two_cond_lemma}
    Suppose one of the following two conditions holds:
    \begin{enumerate}[label = (\roman*)]
        \item the points $c_0 +P_{CM}$ and $c_\infty + P_{CM}$ are not $p$-adically lonely; 
        \item $\tilde{w}$ has a fixed point in $X(\F_p)$ other than the point $\tilde{P}_{CM}$.
    \end{enumerate}
    Then $[D_t], -[D_t] \in W_p$.
\end{lemma}

\begin{proof}
    If (i) holds, then $\tilde{c}_0 + \tilde{P}_{CM}$ and $\tilde{c}_\infty +\tilde{P}_{CM}$ lie in the set $S_p$, and from the definition of $W_p$ in (\ref{W_p}) combined with observation (\ref{D_t_eq}), we see that $[D_t], -[D_t] \in W_p$.

    Suppose instead that (ii) holds, and write $Q \in X(\F_p)$ for the fixed point of $\tilde{w}$ with $Q \neq \tilde{P}_{CM}$. Since $Q$ is a fixed point of $w$, similarly to (\ref{D_t_eq}), we have that 
    \[ 
        ((1-\tilde{w}) \circ \tilde{\iota})(\tilde{c}_0+Q)=\mathrm{red}_p([D_t]) \quad \text{and} \quad 
        ((1-\tilde{w}) \circ \tilde{\iota})(\tilde{c}_\infty+Q)=\mathrm{red}_p(-[D_t]). 
    \]  
    By considering the definition of $W_p$ in (\ref{W_p}), we see that in order to prove the lemma, it will suffice to show that $\tilde{c}_0+Q, \; \tilde{c}_\infty+Q \in S_p$. Suppose the contrary, and assume that $\tilde{c}_0+Q \notin S_p$ (the case $\tilde{c}_\infty+Q \notin S_p$ being identical). From the definition of the set $S_p$, there must therefore exist a known degree $2$ divisor $D = D_1 + D_2 \in X^{(2)}(\Q)$ such that $\tilde{D} = \tilde{c_0} + Q$, and such that $D$ is not $p$-adically lonely. Since $D$ is a \textit{known} divisor in $X^{(2)}(\Q)$, it will either be supported on the set $\{c_0, c_\infty, P_{CM} \}$, or be one of the quadratic points in \Cref{table163} of the introduction.
    
    By \cite[Prosposition~4.3]{ozman}, we have that $Q = \tilde{R}$ for some $R \in X(\Q_p)$ that is a fixed point of $w$. Then $c_0 + R$ and $D$ both reduce to $\tilde{c}_0 + Q$, but since $D$ is $p$-adically lonely, this forces $D = c_0 + R$, so $D_1 = R$ or $D_2 = R$. However, by Lemma~\ref{fixed_points_lemma}, this forces $D_1 = P_{CM}$ or $D_2 = P_{CM} $, a contradiction since $Q \neq \tilde{P}_{CM}$ by assumption.
\end{proof}

We will now see why for certain primes $p$, and in particular for all primes $p < 41$, one of the two conditions in Lemma~\ref{two_cond_lemma} always holds, thus guaranteeing the failure of the sieve when using such primes. We first recall from Section~\ref{sec:difficulties} the notation $K = \Q(\sqrt{-163})$, $\OO = \Z[\sqrt{-163}]$, and $B = \Q(j(\OO))$.

\begin{lemma} \label{lem:ram}
    Let $p \neq 163$ be an odd prime that is not inert in the field $B$. Then one of the two conditions of Lemma~\ref{two_cond_lemma} holds, and consequently $[D_t], -[D_t] \in W_p$. In particular, this holds for any odd prime $p$ that is inert in $K$, which includes all odd primes $<41$.
\end{lemma}

\begin{proof}
    Since $p \notin \{ 2, 163 \}$, the prime $p$ does not ramify in $B$, and since $B$ is a number field of degree $3$, the fact that $p$ is not inert in $B$ implies that there exists a prime $\mathfrak{p} \mid p$ of $B$ of inertia degree $1$. We may therefore choose a representative for the Galois orbit of the point $R_{CM} \in X(B)$ (the fixed point of $w$ defined in Lemma~\ref{fixed_points_lemma}), which we will simply assume to be $R_{CM}$ itself, such that $\tilde{R}_{CM} \in X(\F_p)$. 

    The point $\tilde{R}_{CM} \in X(\F_p)$ is a fixed point of $\tilde{w}$. If $\tilde{R}_{CM} \neq \tilde{P}_{CM}$ then condition (ii) of Lemma~\ref{two_cond_lemma} holds, so suppose instead that $\tilde{R}_{CM} = \tilde{P}_{CM}$. Since $\mathfrak{p}$ has inertia degree $1$ and $p$ does not ramify in $B$, we may embed $B$ in $\Q_p$ and view $R_{CM} \in X(B) \subset X(\Q_p)$. Then $c_0 + R_{CM}$ and $c_0 + P_{CM}$ are different points in $X^{(2)}(\Q_p)$ that both reduce to $\tilde{c_0} + \tilde{P}_{CM}$. It follow that $c_0 + P_{CM}$ is not $p$-adically lonely, and condition (i) of Lemma~\ref{two_cond_lemma} holds by (\ref{equally_lonely}).

    Next, we prove that if $p$ is inert in $K$ then $p$ is not inert in $B$. Recall from Section~\ref{sec:difficulties} that $H$ denotes the ring class field of the order $\OO$. The field $H$ is Galois over $\Q$ with $\Gal(H/\Q)\simeq S_3$. We have $[H:K]=3$ and $[H:B] = 2$. Since there are no primes of inertia degree 6 in $H$, it follows that if $p$ is inert in $K/\Q$, then it cannot be inert in $B/\Q$. 

    Finally, one can either directly check that all odd primes $<41$ are inert in $K$, or apply the following theoretical argument to prove this. Since $K$ has class number $1$, all primes of $K$ are principal, so if a prime $p$ splits in $K$ we must have $p= \alpha \cdot \overline{\alpha}$ where $\alpha = (a+b\sqrt{-163})/2$ for some $a,b \in \Z$. It follows that $\mathrm{Norm}_{K/\Q}(\alpha)=p$, so $a^2+163b^2=4p$. Clearly $b\neq 0$, and it follows that $p\geq \frac{163}{4} > 40$.
\end{proof}

For each odd prime $p < 41$ we checked that $\mathrm{red}_p(P_{CM}) \neq \mathrm{red}_p(R_{CM})$. We also checked that for $3 < p < 41$ the points $c_0 + P_{CM}$ and $c_\infty + P_{CM}$ are $p$-adically lonely (and we expect this to be the case for $p = 3$, but the symmetric Chabauty criterion we apply fails to prove this). So by Lemma~\ref{lem:ram}, the reason for the failure of the sieve using primes $<41$ is due to the cubic fixed point $R_{CM} \in X(B)$. It is interesting how the existence of a cubic point on the curve creates an obstruction to the computation of its quadratic points. 

\begin{remark} 
    We take this opportunity to make a series of technical remarks.
    \begin{itemize}
        \item[(I)] Although any odd prime that is inert in $K$ is not inert in $B$, we note that if a prime splits in $K$, then it need not be inert in $B$. Indeed, the smallest prime that illustrates this behaviour is the prime $167$, which totally splits in $H$, and is therefore totally split in both $K$ and $B$.
        \item[(II)] The contents of Lemma~\ref{lem:ram} and the arguments used in its proof are related to \cite[Section~4]{ozman}.
        \item[(III)] The phenomena we have observed with $X = X_0(163)$ in this section also occur with the curves $X_0(43)$ and $X_0(67)$ (this is due to the fact that $\Q(\sqrt{-43})$ and $\Q(\sqrt{-67})$ are fields of class number $1$), but the analogous prime bound from Lemma~\ref{lem:ram} is smaller, namely $p < 11$ for $X_0(43)$ and $p < 17$ for $X_0(67)$. 
    \end{itemize}
\end{remark}


\section{Incorporating the prime \texorpdfstring{$p = 41$}{p = 41}} \label{sec:41}

In this section, we show how to work with the curve $X$ over the field $\F_{41}$ and we complete the proof of Theorem~\ref{mainthm}.

By Lemma~\ref{lem:35}, to prove \Cref{mainthm} it will suffice to show that for some prime $p$ of good reduction for $X$, we have that $[D_t]$ and $-[D_t]$ are not in $W_p$. By \Cref{lem:ram}, to have any chance of proving this, we must use a prime $p\geq 41$ (with $p \neq 163$) which is not inert in the field $B$. We use the smallest prime that satisfies this condition, namely $p=41$. 

The problem that immediately arises is that the computations in the Atkin--Lehner sieve involving the Jacobian of $X$, as implemented in \cite{akmnov}, become infeasible for such a ``large" $p$. The Jacobian of a genus $g$ curve $Y$ over $\F_p$ is usually represented as the image of a surjective map $Y^{(g)}\rightarrow J(Y)$ sending $D$ to $[D-D_0]$ for some fixed degree $g$ divisor $D_0$ (see \cite[Section~4]{poonen}). In our case ($p=41$, $g=13$), this would require computations over a field of order $41^{13}> 9.2 \cdot 10^{20}$, and hence even the construction of the Jacobian over $\F_{p}$ becomes unfeasible. 

However, we will circumvent this issue by observing that we do not actually need to fully construct $J(\mathbb{F}_p)$. We fix $p = 41$ and first verify that the points $c_0+P_{CM}$ and $c_{\infty}+P_{CM}$ are $p$-adically lonely, meaning that $\tilde{c}_0 + \tilde{P}_{CM}$ and $\tilde{c}_\infty + \tilde{P}_{CM}$ do not lie in the set {$S_p$}. Then, in order to prove that  $[D_t], -[D_t] \notin W_p$, it will suffice to prove that for all  $Q \in X^{(2)}(\F_p) \backslash \{\tilde{c}_0 + \tilde{P}_{CM}, \tilde{c}_\infty + \tilde{P}_{CM}\}$ we have 
\begin{equation} \label{image_cond}
    ((1-\tilde{w}) \circ \tilde{\iota})(Q) \notin \{\mathrm{red}_p([D_t]), \mathrm{red}_p(-[D_t]) \}.
\end{equation}
Now, in order to verify for a given point $Q \in X^{(2)}(\F_p)$ whether $((1-\tilde{w}) \circ \tilde{\iota})(Q) \in J(\F_p)$ is equal to $\mathrm{red}_p([D_t])$ or $\mathrm{red}_p(-[D_t])$, instead of computing the image of the point $Q$ under the map $(1-\tilde{w}) \circ \tilde{\iota}$, it is enough to check whether the divisor $(1-\tilde{w})(Q)$ is linearly equivalent to one of $\tilde{D}_t$ or $-\tilde{D}_t$. This is because $[(1-\tilde{w})(Q)] = ((1-\tilde{w}) \circ \tilde{\iota})(Q)$. This is essentially a computation over the finite field of order $41^2$, instead of a computation involving a field of order $41^{13}$, and it is a computation that is perfectly feasible in \texttt{Magma}. The set $X^{(2)}(\F_p)$ consists of $1598$ degree $2$ divisors, and the total computation time for this step was just over 20 minutes running on a 2200 MHz AMD Opteron. We checked that (\ref{image_cond}) holds, thus proving the following lemma.

\begin{lemma} \label{lem:41}
    The set $W_{41}$ does not contain $[D_t]$ and $-[D_t]$. 
\end{lemma}

We may now simply collate the results of Section~2 with this lemma to prove Theorem~\ref{mainthm}.

\begin{proof}[Proof of Theorem~\ref{mainthm}]
    Combining Lemmas~\ref{lem:35} and \ref{lem:41} proves that  for $\mathcal{P} = \{3,5,41\}$ we have 
    \[ 
        \bigcap_{p \in \mathcal{P}}W_p=\{0 \} \subset J(\Q)_\mathrm{tors}.
    \] 
    As discussed in Section~2, by \cite[Section~3.4]{akmnov}, this proves that every quadratic point on $X$ either arises as the pullback of a rational point on $X/w$ via $\pi$, or is a fixed point of $w$. Applying Lemmas~\ref{quotient_points_lemma} and \ref{fixed_points_lemma} completes the proof of the theorem.
\end{proof}

\begin{remark} \label{end_rem}
    In this section, we worked directly with divisors to check that certain points in $J(\F_p)$ were not equal. By considering the divisors $m \cdot D_t$ for $1 \leq m \leq 27 = \#J(\Q)_{\mathrm{tors}}$, we could have in fact computed the image of a point $Q \in X^{(2)}(\F_p)$ under the map $(1-\tilde{w}) \circ \tilde{\iota}$. Using this idea, for any given prime $p$, we could actually compute the set $W_p$ without fully constructing the Jacobian $J(\F_p)$. This offers a significant computational speed-up when working with larger primes. More generally, this technique of directly working with divisors could be applied to any type of Mordell--Weil sieve, and has the potential to notably reduce computation time and potentially open new avenues for the computation of low-degree points on curves.
\end{remark}


\bibliographystyle{siam}
\bibliography{references}

\end{document}